\documentclass[11pt]{amsart}
\usepackage{amsfonts}

\usepackage[all,cmtip]{xy}
\usepackage{amsmath}
\usepackage{amsthm}
\usepackage{amssymb}
\usepackage{enumitem}       
\newtheorem{thm}{Theorem}[section]

\newtheorem*{definition*}         {Definition}
\newtheorem{lemma}[thm]{Lemma}

\newtheorem{cor}[thm]{Corollary}

\newtheorem*{remark}{Remark}
\theoremstyle{remark}
\newcommand*{\Oo}{\mathcal{O}}

\newcommand*{\Q}{\mathbb{Q}}

\newcommand*{\Z}{\mathbb{Z}}

\newcommand*{\G}{\mathbf{G}}
\newcommand*{\GG}{\mathbb{G}}

\newcommand*{\C}{\mathbb{C}}

\newcommand*{\PP}{\mathbb{P}}

\newcommand*{\ra}{\rightarrow}
\newcommand*{\ol}{\overline}

\def\ker{{\rm Ker}}

\usepackage{amscd,amssymb,amsmath}
\usepackage{color}
\definecolor{purple}{rgb}{0.59, 0.44, 0.84}

\def\hash{\#}
\title{Ax-Schanuel and exceptional integrability}
\author{Jonathan Pila and Jacob Tsimerman}
\begin{document}
\maketitle

\begin{abstract} 
%Given an algebraic function $\alpha(x)$, 
%when can a primitive (i.e. an anti-derivative) 
%of it be constructed by iteratively solving 
%algebraic equations and composing with the 
%primitives of some other given algebraic 
%functions or their inverses? 
When can a primitive of a given algebraic 
function be constructed by iteratively solving 
algebraic equations and composing with the 
primitives of some other given algebraic 
functions or their inverses? 
%iterated applications of: solving algebraic 
%equations, adjoining primitives of some given 
%algebraic functions or their inverse, 
%and compositions? 
We establish some results in this direction.
Specifically, we establish decision procedures 
for determining whether a given primitive can 
be expressed in terms of finitely many others, 
or in terms of elliptic integrals. 

\end{abstract}

\bigskip

\leftline{\bf Contents}

\medskip

1. Introduction and main results

2. Ax-Schanuel for a connected commutative complex 
algebraic group

3. Ax-Schanuel for complex curves and differentials

4. The Seidenberg Elimination Theorem

5. Exceptional integrability in a differential field

6. Hodge-theoretic perspective on meromorphic 
differentials

7. Decision procedures

8. A connection with unlikely intersections

\ \ \, Acknowledgements

\ \ \, References

%{} Other stuff: Decisions, questions, Liouville etc.

\medskip

\section{Introduction and main results}

This paper is concerned with the following 
question. Given an algebraic function $\alpha(x)$, 
when can a primitive (i.e. an anti-derivative) 
of it be constructed by iteratively solving 
algebraic equations and composing with the 
primitives of some other given algebraic 
functions or their inverses? This is the 
exceptional integrability of the title.

\begin{definition*}
Let $B=(B_1,\ldots, B_k)\in \C[X,Y]^k$ be a 
sequence of irreducible polynomials. We say 
that a function $z(x)$, regular on some disk
$\Delta\subset\C$, is \emph{$B$-strictly elementary} 
if there is an open disk $\Delta^*\subset \Delta$ 
and a sequence of pairs
of functions $(x_1, z_1),\ldots, (x_k, z_k)$,
all regular in $\Delta^*$ (as functions of $x$),
such that:

\smallskip

\item{1.} For each $i=1,\ldots, k$,
$B_i\big(x_i, \frac{dz_i}{dx_i}\big)=0$;

\item{2.} For each $i$, either $x_i$ or $z_i$ 
is algebraic over 
$\C\big(x, x_1, z_1, \ldots, x_{i-1}, z_{i-1}\big)$;

\item{3.} $z$ is algebraic over 
$\C\big(x, x_1, z_1, \ldots, x_k, z_k\big)$.

\smallskip

Let $S\subset \C[X,Y]$ be a set of 
irreducible polynomials. We say that a function
$z$ is \emph{$S$-elementary} if there is a 
non-negative integer $k$ and a tuple 
$(B_1,\ldots, B_k)\in S^k$ such that $z$ is 
$(B_1,\ldots, B_k)$-strictly elementary.

\end{definition*}

By an {\it algebraic function\/} we mean a function
$\alpha(x)$, regular on an open disk 
$\Delta\subset \C$, such that there is an algebraic
relation $A(x, \alpha(x))=0$ for $x\in\Delta$. 
Here $A\in \C[X,Y]$ is non-zero and irreducible. 
Thus an $\emptyset$-elementary function is simply 
an algebraic function of $x$. 

An $\{XY-1\}$-elementary function is what is 
classically known as an {\it elementary function\/}, 
as we have $z_i=\log x_i+c$ or 
$x_i=c\exp z_i$, and the question of whether 
a given function has an elementary primitive 
is the question of {\it elementary integrability\/}. 
Elementary integrability is characterized by a 
classical theorem of Liouville 
(see e.g. \cite{ROSENLICHTMONTHLY}; for extensions
see e.g. \cite{SINGERSAUNDERSCAVINESS}). 
Based on this result, Risch \cite{RISCH}
gave an algorithm to decide, in various situations,
whether a given elementary function 
has an elementary primitive 
(see also \cite{DAVENPORT, LAZARD}). 

%By a {\it primitive} of $\alpha(x)$
%we mean a regular function $z(x)$ on some open disk 
%$\Delta_*\subset \Delta$ with $dz/dx=\alpha$, 
%so that in particular $A\big(x, \frac{dz}{dx}\big)=0$. 
%We will also write $z=\int^x\alpha(u)du$.
%%and call $z$ a {\it generalized logarithm\/}.

\begin{definition*}
%{\bf 1.1. Definition.\/}
Suppose that $w(x)$ is a function regular on an 
open disk $\Delta\subset\C$, and 
$B=(B_1,\ldots, B_k)\in \C[X,Y]^k$ and 
$S\subset \C[X,Y]$ as before. We say that $w$ 
is $B$-strictly integrable if some (equivalently 
every) primitive $z$ of $w$ is $B$-strictly
elementary. We say that $w$ is $S$-integrable 
if some (equivalently every) primitive of $w$
is $S$-elementary, and we will say alternatively that
$w$ is integrable in terms of $S$.
%, and that $w$ is exceptionally integrable.

\end{definition*}

We are concerned then with the question of when
an algebraic function $\alpha(x)$, defined by some
$A(x, \alpha(x))=0$, is integrable in terms of
some given set $S$.
%which initially we will take to be finite.
We will see that $(B_1,\ldots, B_k)$-strict 
integrability is really a property of $A$ 
and the $B_i$, not the specific 
branch $\alpha$. More precisely, if some branch 
$\alpha(x)$ of $A(x,y)=0$ on some disk $\Delta$ is 
$(B_1,\ldots, B_k)$-strictly integrable then, 
given any other branch $\alpha'$ and disk 
$\Delta'$, $\alpha'$ is 
$(B_1,\ldots, B_k)$-strictly integrable on (some open 
subdisk $\Delta'_*$ of) $\Delta'$. In fact, suitably
formulated, it is a property of $A, B_1,\ldots, B_k$ 
in any differential field. 

Our first main result is the following, which we 
will establish in a more precise version in the 
setting of a differential field.

\begin{thm}
Suppose that $\alpha(x)$ is an algebraic function  
on a disk $\Delta$ which is 
$(B_1,\ldots, B_k)$-strictly integrable.
Then there is an open disk 
$\Delta_*\subset\Delta$, regular functions 
$x_1,\ldots, x_k, z_1,\ldots, z_k$ and $z$ 
on $\Delta_*$, constants $c_1,\ldots, c_k$, 
and an algebraic function $\gamma(x)$, such that:

\item{1.} $z$ is a primitive of $\alpha$

\item{2.} For $i=1,\ldots, k$, $B_i(x_i, \frac{dz_i}{dx_i})=0$

\item{3.} For $i=1,\ldots, k$, $x_i$ is algebraic over $x$

\item{4.} $z=\sum_{i=1}^k c_iz_i+\gamma$.

\end{thm}

This can be seen as a generalization (to arbitrary
$B_i$) of the restriction of Liouville's theorem 
to algebraic functions $\alpha$. Liouville's Theorem
implies (in particular) the same condition for the
elementary integrability of any elementary function.
The converse clearly holds.

The key observation behind the theorem 
is that $B$-strict integrability implies that 
we are in the exceptional case of the 
Ax-Schanuel theorem for a suitable connected 
commutative complex algebraic group $G$ 
(more specifically, a product of generalized 
Jacobians). We will see further that the linear 
relation  afforded by the theorem is of a 
particular form: it corresponds to a defining 
equation of a coset of a suitable algebraic 
subgroup of $G$. Looking for such a connection 
was suggested by the results of Masser-Zannier
\cite{MASSERZANNIER} (see also \cite{ZANNIERICM})
connecting questions of elementary integrability 
in a pencil of meromorphic differentials on
curves with problems of Zilber-Pink type.

Let $\alpha(x)$ be an algebraic function defined
by $A(x, \alpha(x))=0$.
By supplying some additional algebraic functions 
$\xi=(x,\ldots)$, algebraic over $x$, we can 
associate with $A$ a smooth projective
curve $X$ and meromorphic differential $\omega$ on $X$
with $\omega=\alpha(x)dx$. We can then recast 
the integrability problem in terms of pairs 
$(X, \omega)$ of smooth projective curves and 
meromorphic differentials. This will lead us to a 
geometric reformulation and refinement
of the above result. We will write $z=\int^\xi\omega$
to mean that $z$ is a multi-valued function on $X$
obtained by integrating the form $\omega$. If $x$ 
is a uniformizing function on $X$ then there is 
some associated polynomial $A\in\C[X,Y]$ 
(depending on $X, \omega$, and $x$) such that 
$\omega=\alpha(x)dx$ where
$A \big(x, \alpha(x)\big)=0$. Thus considering $z$ as a 
function of $x$ we have $A(x, dz/dx)=0$.

\begin{definition*}
Let $(X_0, \omega_0)$ and 
$(X_1, \omega_1),\ldots, (X_k, \omega_k)$ be pairs 
of smooth projective curves over $\C$ and meromorphic
differentials.  We say that  $(X_0, \omega_0)$ is 
\emph{$\big((X_1, \omega_1),\ldots, (X_k, \omega_k)
\big)$-strictly integrable} if for some 
(or equivalently any) choice
of non-constant rational functions $x_0, x_i$ of the curves, 
with associated polynomials $A, B_i\in\C[X,Y]$,
setting
$z_0=\int^{\xi_0}\omega, z_i=\int^{\xi_i}\omega_i$, 
we have that $z_0$ is $(B_1,\ldots, B_k)$-strictly
integrable. We define $S$-integrability for finite 
sets of curves and differentials in analogy with 
the definition for algebraic functions.

\end{definition*}

We will prove a reformulation of Theorem 1.1 in 
terms of curves and differentials. This result 
(Theorem 3.3) is stated and proved in section 3. 
The algebraic relations among the uniformizing 
variables gives a curve 
$Z\subset X_0\times X_1\times\ldots X_k$ giving
a correspondence between $X_0$ and each $X_i$. This
allows us to pull back differentials 
from $X_i$ to $Z$ and to express integrability 
in terms of their traces to $X_0$.

\begin{definition*}

Let $(X_i,\omega_i), i=1,2$ be curves with 
differentials. Let $Z\subset X_1\times X_2$ be a 
curve with no fibral-components. Then if
$\omega_2=\pi_{2*}\pi_1^* \omega_1$ we say 
that $\omega_2$ is a \emph{trace-image} of 
$\omega_1$. 

\end{definition*}

\begin{thm}
Let $(X_i,\omega_i), 0\leq i\leq k$ be smooth
projective curves with meromorphic differentials. 
The following conditions are equivalent.

\smallskip
%\begin{itemize}

    \item{1.} $(X_0,\omega_0)$ is integrable in terms 
    of $\{(X_1,\omega_1), \ldots, (X_k, \omega_k)\}$
    
    \item{2.} $\omega_0$ is in the linear span of differentials on $X_0$ which are trace-images of the $\omega_i$, and exact differentials.
    
%\end{itemize}

\end{thm}

This theorem can then serve as the basis of a 
decision procedure when all the curves and
differentials are defined over an explicitly given
finite type field. This generalizes to general sets $S$
the special case of Risch's algorithm in which
$\alpha$ is algebraic rather than just elementary
(on this see \cite{DAVENPORT}).

\begin{thm}
%[Geometric version]
Let $K$ be a finite type field of characteristic zero.
There is a decision procedure with the following
property. Given pairs $(X_0, \omega_0)$ and
$(X_1, \omega_1),\ldots, (X_k, \omega_k)$ 
of smooth projective complex curves and meromorphic
differentials on them, defined over $K$,
the procedure decides  whether $(X_0, \omega_0)$ 
is integrable in terms of
$\{(X_1, \omega_1), \ldots, (X_k, \omega_k)\}$.

\end{thm}

Hardy, in his book \cite{HARDY} on the integration
of univariate functions, discusses (p9-10) the known
results concerning the integration of algebraic 
functions, in particular concerning whether they
are or are not elementary. He notes (p10) that 
there are cases in which ``integrals associated 
with curves whose
deficiency [i.e genus] is greater than unity are 
in reality reducible to elliptic integrals'', and
observes that ``no general method has been
devised by which we can always tell, after a 
finite series of operations, whether any given
integral is really elementary, or elliptic, or
belongs to a higher order of transcendents''. 
A precise definition of ``reducible to elliptic
integrals'' is not given, but the following could
be considered as providing a procedure for a 
natural formulation of this question.

\begin{thm}%\label{ellipticdecision}
Let $K$ be a finite type field of characteristic 
zero. Let $(X,\omega)$ be a curve with a rational 
differential form, defined over $K$. Then there 
is a decision procedure for determining whether
$(X,\omega)$ is integrable in terms of the set 
of all elliptic curves (over $\ol K$) with all rational 
differentials (over $\ol K$) on them.
\end{thm}

The plan of the paper is as follows. In \S2 
we prove Ax-Schanuel for a commutative 
complex algebraic group,
generalizing Bertrand's generalization \cite{BERTRAND}
of results of Ax \cite{AX, AXSOME} (see also Kirby 
\cite{KIRBY}). In \S3 we recall the definition 
and some properties of {\it generalized Jacobians\/}
and reformulate exceptional integrability in terms 
of differentials on curves. In \S4 we recall the 
statement of the Seidenberg Elimination Theorem.
Then in \S5 we reformulate both exceptional 
integrability and Ax-Schanuel in the setting of 
a differential field and prove a version
of Theorem 1.1 in that setting, making extensive 
use of \S4. \S6 interprets meromorphic forms on a 
curve in terms of cohomology. Theorem 1.2 is 
then proved in \S7 followed
by the description of the decision procedures.
The final section \S8 connects integrability
questions for curves with regular differentials 
in a pencil with certain problems of unlikely
intersections.

\section{Ax-Schanuel for a  connected commutative 
complex algebraic group}

Let $G$ be a connected commutative complex
algebraic group. Recall that such a group is 
an extension of a semiabelian variety by a vector 
group. Following Bertrand \cite{BERTRAND}, say 
that such a group has {\it no vectorial quotients} 
if there do not exist non-trivial maps
$G\rightarrow\G_{\rm a}$. We begin with the 
following simple (and certainly well known) lemma:

\begin{lemma}
Let $G$ be a connected commutative complex 
algebraic group. There is a canonical connected 
subgroup $H$ of $G$ without vectorial quotients, 
such that $G\cong H\times V$ for some 
(non-canonical) vector subgroup $V<G$.
\end{lemma}

\begin{proof}
We first establish uniqueness. Indeed, such an $H$ 
would necessarily have to be the maximal subgroup 
of $G$ which lies in the kernel of every map to 
$\G_{\rm a}$. 

To show existence, define $H$ to be the maximal 
subgroup of $G$ which lies in the kernel of every 
map to $\G_{\rm a}$. Note that since $G$ is 
connected and  $\G_{\rm a}$ is simply 
connected, $H$ must also be connected.

Next, let $\phi:G\rightarrow V$ be a map onto a 
vector group  such that $H=\ker\phi$. By replacing 
$V$ by the image of $\phi$ we may assume $\phi$ is
surjective. It remains to construct a section of 
$\phi$. 

Let $U$ be the unipotent radical of $G$. We claim
$\phi(U)=V$. Indeed, if not, we would obtain a
surjection map from the semiabelian variety $G/U$ 
to the vector group $V/\phi(U)$. Finally, as maps
between vector spaces split we may find a section 
of $\phi\mid U$, as desired.
\end{proof}

We will call $H$ the {\it maximal NVQ subgroup\/} 
of $G$. Note that, if $H_i$ are the maximal NVQ
subgroups of $G_i$, $i=1,\ldots, k$, then
$H_1\times\ldots\times H_k$ is the maximal
NVQ subgroup of $G_1\times\ldots\times G_k$.
We proceed to prove our main theorem of this section:

\begin{thm}[Ax-Schanuel for a connected commutative
complex algebraic group]
Let $H$ be a complex connected commutative algebraic 
group without vectorial quotient, and let $V$ be a 
vector group. Let $\pi:T\rightarrow H$ be the universal
cover of $H$ with graph $D$. Let 
$Z\subset V\times T\times H$ be an algebraic variety 
and $U\subset Z\cap (V\times D)$ be an analytically 
irreducible component. If $\dim U>\dim Z-\dim H$ 
then the projection of $U$ to $H$ lies in the translate 
of a proper subgroup (i.e. a proper weakly special).
\end{thm}

\begin{proof}
The case where $\dim V=0$ is proven by Bertrand \cite{BERTRAND}
We shall reduce to that case. Without loss of generality we can 
assume that $Z$ is the Zariski closure of $U$.  

Let $Z'$ be the projection of $Z$ to $T\times H$ and let $U'$ be the 
projection of $U$ to $D$. Then $U'$ is a component of $Z'\cap D$. 
Now since $Z$ is the Zariski closure of $U$, the generic dimension 
of the fibers of the  projection $Z\rightarrow Z'$ is the same as the 
generic dimension of the fibers of the projection $U\rightarrow U'$. 
Thus $\dim Z-\dim U = \dim Z'-\dim U'$. The claim therefore follows 
immediately from the case where $\dim V=0$.
\end{proof}

It might seem that the above is artificial, and so we explain how to 
relate it to the usual geometric Ax-Schanuel statement.
Indeed, let $G$ be an arbitrary complex connected commutative 
algebraic group, and let $H$ be the subgroup without vectorial 
quotient provided by the lemma, so that $G=H\times V.$

Now let $T_G,T_H,T_V$ be the tangent spaces of $G,H,V$ 
respectively. The key point is that the map 
$\pi_V:T_V\rightarrow V$ is an algebraic isomorphism. 
Therefore, if we let $D_H,D_G,D_V$ be the graphs, 
then while  $D_H^{\rm zar} = T_H\times H,$ 
we have $D_V^{\rm zar}=D_V$, and therefore 
$$
D_G^{\rm zar}=
D_H^{\rm zar}\times D_V^{\rm zar}\cong T_H\times H\times V.$$

Now if we are to set up the usual Ax-Schanuel 
as $Z\subset T_G\times G$ and $U$ a component 
of $Z\cap D_G,$ then the assumption of 
$Z=U^{\rm zar}$ would entail that 
$Z\subset D_G^{\rm zar}$ and thus this is the
natural space to work in.

The projection to $H$ is non-canonical, but the
dimension of a minimal coset containing the image of
$U$ is independent of the choice of projection.
The coset can (non-canonically) be given a group
structure and we can reframe the theorem in the
following way.

\begin{cor}
%{\bf 2.3. Corollary.\/}
Let $H$ be a complex connected commutative algebraic 
group without vectorial quotient, and let $V$ be a 
vector group. Let $\pi:T\rightarrow H$ be the universal
cover of $H$ with graph $D$. Let 
$Z\subset V\times T\times H$ be an algebraic variety 
and $U\subset Z\cap (V\times D)$ be an analytically 
irreducible component.
Let $K$ be a minimal coset containing the image of $U$
under projection to $H$. Then
$$
\dim U\le \dim Z - \dim K.
$$
\end{cor}

\section{Ax-Schanuel for complex curves 
and differentials}

Much of the following follows \cite{SERRE}, 
but we work throughout over $\C$.

Let $X=X(\C)$ be a smooth projective complex algebraic curve,
$\Sigma\subset X$ a finite set of (distinct) points, and
${\bf m}$ a {\it modulus\/} with support $\Sigma$ (i.e. an 
assignment of a positive integer $n(P)$ to each $P\in\Sigma$).

Associated with this data is the generalized Jacobian 
$J_{\bf m}=J_{\bf m}(X)$, with
$g_{\bf m}=\dim J_{\bf m}=\dim H^0\big(X, \Omega({\bf m})\big)$.
We have a map
$$
\big(X-\Sigma\big)^{g_{\bf m}}\rightarrow J_{\bf m}
$$
sending $(x_1,\ldots, x_g)$ to the corresponding divisor class
(following the choice of a suitable base point),
and which factors through the symmetric product
$\big(X-\Sigma\big)^{(g_{\bf m})}$. The map 
$\big(X-\Sigma\big)^{(g_{\bf m})}\rightarrow J_{\bf m}$
is birational. 
If we map $X-\Sigma\rightarrow J_{\bf m}$ sending $x$ 
to the corresponding divisor class then the image generates 
$J_{\bf m}$ as a group.

Let  $T_{\bf m}$ be the lie algebra of $J_{\bf m}$,
identified with its tangent space at the origin.
Here see particularly \cite{SERRE} pages 100--101.
See also \cite{MUMFORD}, pages 46--47.
We have the map
$$
\big(X-\Sigma\big)^{g_{\bf m}}\rightarrow T_{\bf m}
$$
given by integrating the differentials in 
$H^0\big(X, \Omega({\bf m})\big)$ (again following the 
choice of a suitable base point in $X-\Sigma$), and this map
composed with the Lie exponential recovers the
divisor class map, if we stipulate that the tuple 
of base points maps to the identity.
The generalised Jacobian $J_{\bf m}$ is analytically
isomorphic to the quotient of $T_{\bf m}$ by the 
group $\Pi_{\bf m}$ of periods of the differentials.

\begin{lemma}
The group $J_{\bf m}$ has a non-trivial 
vectorial quotient if and only if there 
exists a non-trivial exact differential in
$H^0\big(X, \Omega({\bf m})\big)$. 
\end{lemma}

\begin{proof}
Let $V=H^0\big(X, \Omega({\bf m})\big)$.
We have  $J_m(\mathbb{C}) = V^\vee/\Pi$ where $\Pi$ 
is the group of periods. Then
$J_{\bf m}$ having a non-trivial vectorial quotient 
is the same as saying the complex span of $\Pi$ 
is not all of $V^\vee$, which is the same as 
saying there is a non-zero $v\in V$ which vanishes 
on $\Pi$, which is the same as saying there is a
non-zero differential whose periods vanish. 
But given such a differential, we can integrate 
it globally and therefore its exact. 
\end{proof}

Further, the dimension of the largest vectorial
quotient is equal to the dimension of the subspace
$E_{\bf m}$ of $H^0\big(X, \Omega( {\bf m})\big)$ 
of exact ${\bf m}$-differentials. Let $H_{\bf m}$ 
denote the maximal NVQ subgroup of $J_{\bf m}$, 
of dimension $h_{\bf m}$.

Now we consider a finite sequence of pairs
$(X_1, \omega_1), \ldots, (X_k, \omega_k)$ 
where $X_i$ is a smooth projective complex 
curve and $\omega_i$ is a 
non-zero meromorphic differential. 
For $i=1,\ldots, k$, let ${\bf m}_i$ be a modulus 
on $X_i$ with 
$\omega_i \in H^0\big(X_i, \Omega({\bf m}_i)\big)$.
We adopt the above notation, but replace subscripts 
${\bf m}_i$ by $i$. Thus we have the corresponding
generalised Jacobians $J_i=J_{{\bf m}_i}(X_i)$, of
dimension $g_i$, their maximal NVQ subgroups 
$H_i=H_{{\bf m}_i}$, etc.

Let $G=J_1\times\ldots\times J_k$ be the product 
of generalised Jacobians.
Let $T=T_1\times\ldots\times T_k$ its Lie algebra.
Then  $H=H_1\times\ldots\times H_k$ is the 
maximal NVQ subgroup of $G$, of dimension $h=\sum h_i$.
%Likewise $L=L_1\times\ldots\times L_k$ is the 
%maximal NVQ subgroup of $T$. {\bf Do I need to 
%choose complements??? Choosing a basis of 
%differentials should be enough???}
%We choose some vectorial complements $V_i$ so that
%$J_i=H_i\times V_i$, and then $V$ is a vectorial complement
%of $H$ in $G$, which gives a notion of projections. The map of 
%$V_i$ to its Lie algebra is an algebraic isomorphism and we can
%project to $L$. {\bf Fix???}

We choose a basis  $\omega_{ij}, j=1, \ldots ,g_i$ 
of $H^0\big(X_i, \Omega({\bf m}_i)\big)$, 
with $\omega_i=\omega_{i1}$. 
Pick base points and define locally
$$
z_{ij}=\int^{x_i}\omega_{ij}, \quad 
i=0,\ldots, k,\quad j=1,\ldots, g_i
$$
where $x_i\in X_i(\C)$. 
We will also write $z_i=z_{i1}$.

By a {\it locus\/} in a variety (or complex space) 
$W$ we will mean a regular map 
$w: \Delta \rightarrow W$ on
some open disk in the complex plane. 
If we have loci $x_i: \Delta \rightarrow X_i-\Sigma_i$
for $i=1,\ldots, k$ then the divisor class map to
the generalized Jacobian, and integration, give a locus
$$
(x,z): \Delta \rightarrow G \times T.
$$
Here $x=(x_1,\ldots, x_k)$ while 
$z=(\overline{Z_1},\ldots, \overline{Z_k})$
with $\overline{Z}_i=(z_{i1},\ldots, z_{i {g_i}})$.
Let $U$ denote the image of the locus and $Z$ its
Zariski closure. Then
$$
\dim Z= {\rm tr. deg.}\big({\C}
(x_i,z_{ij}, i=1,\ldots, k, j=1,\ldots, g_i)/\C\big).
$$

Later we will assume that each
$\omega_i$ is not exact. It will then
be convenient to choose a basis $\omega_{ij}$
of $H^0\big(X_i, \Omega({\bf m}_i)\big)$ in
such a way that 
$\omega_{i1},\ldots, \omega_{i{h_i}}$
is a basis of 
$H^0\big(X_i, \Omega({\bf m}_i)\big)/E_{i}$.
With such a choice, we note that the condition 
that the  projection $U$ to $H$ lies in a 
coset of a proper subgroup is equivalent 
to the existence of
constants  $c_{ij}$, not all zero,
primitives $\gamma_i$ of 
exact ${\bf m}_i$-differentials for 
$i=1,\ldots, k$, and a constant $c$, such that
$$
\sum_{i=1}^k\sum_{j=1}^{h_i} c_{ij}z_{ij}+\sum_{i=1}^k \gamma_i=c.
$$
where this corresponds to a defining relation of 
a coset of an irreducible algebraic subgroup of $H$. 
Thus the linear relations given by the theorems
are of a particular type, corresponding
to suitable ``bi-algebraic'' varieties under the
uniformisation of the group $G$.
%We will say that such a relation is of  
%{\it weakly special type\/}.

\begin{thm}[Ax-Schanuel Theorem for complex curves
and differentials]\label{thm:complex-axs}

Let $(x,z)$ be a locus in $G\times T$ 
with image $U$ as above with each $x_i$ non-constant.
Suppose that
$$
{\rm tr. deg.}\big({\C}(x_1,\ldots, x_k, z_1,\ldots, z_k)/\C\big)\le k
$$
Then there is a non-empty subset
$I\subset\{1,\ldots, k\}$,
non-zero complex constants $c_i$, and 
primitives $\gamma_i$ of exact 
${\bf m}_i$-differentials on $X_i$ for 
$i\in I$, such that, as functions of $t\in\Delta$,
$$
\sum_{i\in I} c_i(z_i+\gamma_i)=0.
$$
(Note that the degree of $\gamma_i$ is bounded in terms of  $X_i, {\bf m}_i$). Further, this relation corresponds
to a defining relation of a coset of an algebraic
subgroup of $H$.

Suppose $I$ above has 
$\hash I=\ell$ and rename the 
corresponding curves and differentials as
$(Y_i, \eta_i), i=1,\ldots, \ell$.
Then there exists an algebraic curve in 
$Y_1 \times\ldots\times Y_\ell$,
dominant to each factor, such that, writing
the generic point as $(\xi_1, \ldots, \xi_\ell)$, 
and writing $\zeta_i=\int^{\xi_i}\eta_i$, 
$i=1,\ldots, \ell$, we have
$$
\sum_{i=1}^\ell c_i\zeta_i=\gamma(\xi_1)
$$
with the same complex constants $c_i$ as above
and an algebraic function $\gamma(\xi_1)$. 

\end{thm}

\begin{proof}
If some $\omega_i$ is exact then we immediately
get the conclusion with $I=\{i\}$. 
So we assume that all the $\omega_i$ are inexact. 

Then we can complete each $\omega_{i1}$ to a basis  
$H^0\big(X_i, \Omega({\bf m}_i)\big)/E_i$ by 
adjoining a further $h_i-1$ of the $\omega_{ij}$. 
The remaining $z_{ij}$ are algebraic over these 
and the $x_i$. We thus find
$$
{\rm tr. deg.}
\big({\C}(x_i, z_{ij})/\C\big)\le h.
$$

By Ax-Schanuel, since $\dim U=1$, the locus $U$ 
projects into a proper coset in $H$.
As noted above, this is equivalent to the $z_{ij}$ 
satisfying a linear relation of a specific type: 
there exist constants  $c_{ij}$, not all zero,
primitives $\gamma_i$ of 
exact ${\bf m}_i$-differentials for 
$i=1,\ldots, k$, and a constant $c$, such that
$$
\sum_{i=1}^k\sum_{j=1}^{h_i} c_{ij}z_{ij}+\sum_{i=1}^k \gamma_i=c,
$$
where this relation corresponds to a defining 
relation of a coset in $H$.

Let $K$ be a minimal such proper coset, whose
codimension in $H$ is $\ell$.
Then the $z_{ij}$ satisfy $\ell$ independent 
relations as above (meaning that the vectors 
$(c_{ij})$ are linearly independent over $\C$). 

Let us suppose that  $z_{11},\ldots, z_{k1}$ are 
linearly independent modulo primitives of exact
${\bf m}_i$-differentials. Then we can complete 
$z_{11},\ldots, z_{k1}$ to a basis of
the $z_{ij}$ modulo such relations of size 
$h-\ell$, where $h=\sum_{i=1}^k h_i$, by including a 
further $h-\ell-k$ of the $z_{ij}$. 

We now apply Ax-Schanuel to the group 
$G'=K\times V$ and its
Lie algebra $L'\times M$. We have $\dim U=1$ and
$$
{\rm tr. deg.}\big(\C(x_1,\ldots, x_k, z_{11},
\ldots, z_{kg_{k}}/\C\big) \le h-\ell
$$
and we conclude that the locus projects into 
a proper coset in $K$. This contradicts the 
mimimality of $K$, and we conclude that
$z_1,\ldots, z_k$ cannot be linearly independent 
modulo primitives of exact ${\bf m}_i$-differentials.
Therefore  we get some relation $R$ 
of the required form.
The degree of $\gamma_i$ is bounded, once $X_i$
and ${\bf m}_i$ are specified, as this fixes
the spaces of exact differentials.

Let $I$ be the set of indices for which $c_i$ 
in $R$ is non-zero. Now $G'$ is dominant to each 
factor $J_i, i=1,\ldots, k$, as $x_i$, being 
non-constant, is a generic point of $X_i$ over 
$\C$ and $X_i$ generates $J_i$ as a group. 
Fix $i\in I$ and fix points $y_j\in X_j$ for 
$I\ni j \ne i$. Consider the coset $K_i\subset J_i$ 
of points $y_i$ such that $\overline{y}\in G'$.
If $K_i$ contains $X_i$ then it contains all $J_i$. 
But this is impossible, as fixing $y_j, j\ne i$ 
entails a restriction on $z_i$. Thus we see that, 
for each $i\in I$, $x_i$ is algebraically 
dependent on $\{x_j: j\in  I, j \ne i\}$. 

Now taking $I$ as above, and renumbering, we can 
take an algebraic curve $Z$ in $G'\cap Y$,
where $Y=\prod Y_i$, such that each $Y_i$ 
coordinate is algebraic over each other $Y_j$.
\end{proof}

Now we consider exceptional integrability. 
Suppose that $(X_0, \omega_0)$ and
$(X_1,\omega_1),\ldots,(X_k, \omega_k)$
are smooth projective complex curves with 
meromorphic differentials as above. We keep all 
the above notation, in particular we have
moduli ${\bf m}_i$ on $X_i$ such that
$\omega_i\in H^0\big(X_i, \Omega({\bf m}_i)\big)$,
$i=0,\ldots, k$.

\begin{thm}[Exceptional integrability for 
complex curves, differentials] \label{thm:ei}
%Suppose that  $x_0: \Delta\rightarrow X_0$ and
%$x_i: \Delta\rightarrow X_i,  i=1,\ldots, k$ 
%with $x_0$ non-constant. Suppose 
%$z_0, z_1,\ldots, z_k$ are regular for 
%$t\in\Delta$, and we have the following properties.
%
%\smallskip
%
%\item{1.} $z_0=\int^{x_0}\omega_0$ on $X_0$
%
%\smallskip
%
%\item{2.} For each $i=1,\ldots, k$, 
%$z_i=\int^{x_i}\omega_i$ on $X_i$
%
%\smallskip
%
%\item{3.} For each $i$, either $x_i$ 
%or $z_i$ is algebraic over 
%$\C(x_0, x_1, z_1,\ldots, x_{i-1}, z_{i-1})$
%
%\smallskip
%
%\item{4.} $z_0$ is algebraic over 
%$\C(x, x_1, z_1,\ldots, x_k, z_k)$.
%
%\smallskip
%\noindent
%Then there is a disk $\Delta^*\subset \Delta$
%and regular functions $\xi_i, \zeta_i, i=1,\ldots k$
%on $\Delta^*$ with the following properties.
%
%\smallskip
%
%\item{${\it 2}^*$.} For each $i=1,\ldots, k$, 
%$\zeta_i=\int^{\xi_i}\omega_i$ on $X_i$
%
%\smallskip
%
%\item{${\it 3}^*$.} For each $i$, $\xi_i$ 
%is algebraic over  $\C(x_0)$
%
%\smallskip
%
%\item{${\it 4}^*$} There are constants $c_i$ and 
%an algebraic functions $\gamma_i$ such that
%$$
%z_0=\sum_i c_iz_i+\gamma_i(x_i),
%$$
%where this relation is of weakly special type.
%The degrees of the $\gamma_i$ are bounded for given
%curves $X_i$ and moduli ${\bf m}_i$.
%
%\smallskip
%
%Let $(X, \omega)$ be a smooth projective 
%curve with a meromorphic differential, and let
%$(\{(X_1, \omega_1),\ldots, (X_k, \omega_k))$
%be a finite sequence of smooth projective curves 
%and meromorphic differentials.
%
Suppose that $(X_0, \omega_0)$ is 
$((X_1, \omega_1),\ldots, (X_k, \omega_k))$-strictly
integrable. Then (after possibly replacing 
$((X_1, \omega_1),\ldots, (X_k, \omega_k))$ with a 
subsequence) there exists a curve 
$$
Z\subset X_0\times X_1\times\ldots\times X_k,
$$
irreducible and dominant to each factor such that,
writing $(\xi_0, \xi_1,\ldots, \xi_k)$ for a  
point of $Z$ avoiding the supports of
all the ${\bf m}_i$, and $z_0=\int^{\xi_0}\omega_0$,
$z_i=\int^{\xi_i}\omega_i$,
there exist constants $c_i$ and primitives 
$\gamma_i(\xi_i)$ of exact ${\bf m}_i$-differentials
such that, locally,
$$
z_0=\sum_{i=1}^k c_i z_i+\gamma_i.
$$

\end{thm}

\begin{proof} If $\omega_0$ is exact the conclusion
is immediate. So we assume it is inexact.
If some $\xi_i$ is constant then $\xi_i, z_i$ 
are algebraic over $\C$, and can be removed 
from the sequence  preserving the integrability.
%properties 1, 2, 3. 
So we may assume all $\xi_i$ are non-constant.
If $\omega_i$ is exact for some 
$i\in\{1,\ldots, k\}$ then both $\xi_i, z_i$ are 
algebraic over previous 
and we can omit them from the sequence.
So we may assume all $\omega_i$ are inexact.

The transcendence degree of 
$\C(x_0, \xi_1, z_1,\ldots, \xi_k, z_k, z_0)$
over $\C$ is at most $k+1$. Thus, by Theorem 3.2, 
we get a non-empty set of indices 
$I\subset\{0, 1, \ldots, k\}$ with the properties
given there.

Suppose $I$ does not include the index 0. Then, 
taking the highest index $i\in I$, we see that $z_i$ 
and $\xi_i$ are both algebraic over the 
$\{\xi_j, z_j, j\in I, j\ne i\}$. So we can omit 
$\xi_i, z_i$ from the sequence.

So eventually we find that $I$ does include $0$. 
Then renumbering the curves and taking the
curve in the product, dominant to each factor, as 
in Theorem 3.2, gives the required conclusion.
\end{proof}

Again, the linear relation corresponds to a 
defining relation
of a coset of an irreducible algebraic subgroup
of the maxinal NVQ subgroup $H$ of $G$.

The idea will be to leverage the above theorem 
into a general statement in a differential field 
(and with rationality consequences) using the 
Seidenberg Elimination Theorem together with 
the Seidenberg Embedding Theorem.

\section{The Seidenberg Elimination Theorem}

Let $K$ be a differential field with commuting derivations 
$D_1,\ldots, D_m$ and constant field $C$.
The differential polynomial ring $K\{U_1,\ldots, U_n\}$ in $n$
differential indeterminates $U_i$ is the differential ring obtained
by the adjunction of the $U_i$. Each differential polynomial
$f\in K\{U_1,\ldots, U_n\}$ gives rise to a function $f: K^n\rightarrow K$
by substitution.

\begin{thm}[Seidenberg Elimination Theorem \cite{SET}, 
as in \S1.6 of \cite{BUIUMCASSIDY}]
Given a finite system
$$
f_1(a_1,\ldots, a_N, y_1,\ldots, y_n)=0, \ldots, 
f_s(a_1,\ldots, a_N, y_1,\ldots, y_n)=0,\leqno{1}
$$
$$
g(a_1,\ldots, a_N, y_1,\ldots, y_n)\ne 0,
$$
where 
$f_1,\ldots, f_s, g\in\Z\{a_1,\ldots, a_N, y_1,\ldots, y_n\}$, 
there are a finite number of finite systems
$$
f_{i1}(a_1,\ldots, a_N)=0, \ldots, f_{is_{i}}(a_1,\ldots, a_N)=0, \quad
g_{i}(a_1,\ldots, a_N)\ne 0, \leqno{1'}
$$
where $f_{ij}, g_i\in \Z\{a_1,\ldots, a_N\}$,
such that for any differential field $K$, 
and for any $\overline{a}_1,\ldots, \overline{a}_N\in K$,
the system
$$
f_1(\overline{a}_1,\ldots, \overline{a}_N, y_1,\ldots, y_n)=0, \ldots, 
f_s(\overline{a}_1,\ldots, \overline{a}_N, y_1,\ldots, y_n)=0,\leqno{\overline{1}}
$$
$$
g(\overline{a}_1,\ldots, \overline{a}_N, y_1,\ldots, y_n)\ne 0
$$
has a solution in some differential extension field of $K$ if and only if
there is an $i$ such that 
$$
(\overline{a}_1,\ldots, \overline{a}_N)
$$
solves ${1}'$. Further, there is an algorithm that, given the system $1$,
produces the finite systems ${1}'$.\ \qed
\end{thm}

On the complexity of this algorithm 
see \cite{LEONSANCHEZOVCHINNIKOV}
and the references therein. In the case
of one derivation see \cite{OVPOVO}.

\section{Exceptional integrability in a 
differential field}

Let $K$ be a differential field of characteristic 
zero with derivation $D$ and constant field $C$.
Let $x\in K$ be non-constant. We assume $Dx=1$ 
(this can always be achieved by replacing the given
derivation $D$ by $\frac{1}{Dx}D$).

Let $\xi, \alpha\in K$. By saying that $\alpha$ is 
an algebraic function of $\xi$ we mean that 
$A(\xi, \alpha)=0$ for some non-zero absolutely 
irreducible $A\in C[X, Y]$. By saying that $\alpha$ 
is exact (relative to $x$) we mean that there exists 
$z$ algebraic over $x$, perhaps in some differential
extension field, with $\frac{Dz}{Dx}=\alpha$.

The next result is a version of Theorem 3.2 in a
differential field, and it will include the statement 
of Theorem 1.2. 

\begin{definition*}
Let $K$ be a differential field of characteristic 0 
with constant field $C$ and $x\in K$ with $Dx=1$. 
Let $A, B_1,\ldots, B_k\in C[X,Y]$ be absolutely
irreducible. Let $\alpha\in K$ with $A(x, \alpha)=0$.
We say that $\alpha$ is $(B_1,\ldots, B_k)$-strictly
integrable if, in some differential 
extension field $\tilde{K}$
(with constant field $\tilde{C}$), there exist elements
$x_1,\ldots, x_k, z_1,\ldots, z_k, z$ with the following
properties:

\smallskip

\item{1.} $Dz=\alpha$

\smallskip

\item{2.} For $i=1,\ldots, k$ we have 
$B_i(x_i, \frac{1}{Dx_i}Dz_i)=0$

\smallskip

\item{3.} For $i=1,\ldots, k$ either $x_i$ or 
$z_i$ is algebraic over 
$\tilde{C}(x, x_1, z_1,\ldots, x_{i-1}, z_{i-1})$

\smallskip

\item{4.} $z$ is algebraic over 
$\tilde{C}\big(x, x_1, z_1, \ldots, x_k, z_k\big)$.
\end{definition*}

%\begin{definition*}
%Let $G$ be a commutative connected algebraic group
%over $C$, and $x\in G(K)$. We will say that $x$ 
%lies in an NVQ-proper coset of $G$ if, over some
%differential extension field, there is a coset 
%$K$ of $G$ defined over $\tilde{C}$ such that 
%$x\in K(\tilde{K})$ and $H(K)$ is a proper coset 
%of $H(G)$.
%
%\end{definition*}

\begin{thm}[Theorem on exceptional integrability 
in a differential field]
Let $K$ be a differential field of characteristic 0 
with constant field $C$ and $x\in K$ with $Dx=1$. 
Let $A, B_1,\ldots, B_k\in C[X, Y]$ be absolutely irreducible. 
Let $\alpha\in K$ with $A(x, \alpha)=0$.

Suppose $\alpha$ is $(B_1,\ldots, B_k)$-strictly 
integrable. 

%\smallskip
%
%\item{(i)} The locus $x, x_1,\ldots, x_k$ lies
%in an NVQ-proper coset of the product $G$ of
%generalized Jacobians.
%
%\smallskip
%
%\item{(ii)} There is a non-empty subset
%$I\subset\{1,\ldots, k\}$, constants 
%$c_i, i\in I$ all not zero, such that
%$$
%\sum c_i (z_i+w_i)=0
%$$
%where $w_i$ are primitives of 
%${\bf m}_i$-exact differentials,
%whose degree over $x_i$ is bounded as in 3.2.
%Moreover, the $w_i$ as algebraic functions 
%of $x_i$ and the constants $c_i$ can be taken 
%to be defined over an algebraic extension of
%$\Q(\overline{a}, \overline{b}_i)$.
%
%\smallskip

Then there exists a differential extension 
field $\tilde{K}$ of $K$ (with constants $\tilde{C}$)
containing elements $x_i$ algebraic over $x$,
constants $c_i$, and algebraic functions 
$\gamma_i(x_i)$ of degree bounded as in Theorem 3.2 
such that

\smallskip

\item{1.} $Dz=\alpha$

\smallskip

\item{2.} For each $i=1,\ldots, k$, 
$B_i(x_i, \frac{1}{Dx_i}Dz_i)=0$

\smallskip

\item{3.} For each $i=1,\ldots, k$, $x_i$ is 
algebraic over $C(x)$

\smallskip

\item{4.} $z=\sum_i c_iz_i+\gamma_i(x)$.

\smallskip
\noindent
\item{}
Moreover, the constants, the $x_i$ as algebraic 
functions of $x$ and the $\gamma_i(x_i)$ can be 
taken to be defined over an algebraic extension 
of $\Q(\overline{a}, \overline{b}_i)$, where
these are the tuples of coefficients of $A, B_i$.

%\smallskip
%
%\item{(i)'} The proof shows how to define a notion 
%of ``weakly special subvariety'' of 
%$X\times X_1\times\ldots\times X_k$
%with respect to the given differentials
%so that: in case of exceptional integrability, 
%the $(x, x_i)$ lie in such a locus. 
%{\bf Something more precise??? Separate this 
%and state an Ax-Schanuel for curves and differentials
%with a notion of ``weakly special subvariety''???}
\end{thm}

\begin{proof}
After extending the system with some further 
elements algebraic over $x$ and each $x_i$ to get 
tuples $\overline{x}, \overline{x_i}$, they describe
algebraic curves $X_i$ whose projective closures are
smooth.
After a Seidenberg embedding, the $z, z_i$ 
(perhaps adjusted by complex constants) are integrals 
of meromorphic differentials $\omega, \omega_i$ on 
these curves.

So we are in the situation of Theorem 3.2, and
we conclude that there is a linear relation 
as in 3.2, with some algebraic $\gamma_i$ of degree 
at most $d_i$ over $x_i$. This is then a linear 
relation in some differential extension field of $K$.

Let $\overline{a}, \overline{b}_i$ be the coefficient
vectors of $A, B_i$. Under the hypotheses, there is a
choice of degrees of the algebraic relations at stage $i$,
whether it is $x_i$ or $z_i$ that is algebraic over
previous elements, and the degree of the final 
algebraic relation of $z$ over the other elements.
Call this choice of discrete data a {\it shape\/}.
Given a shape $\tau$, the connecting
algebraic relations $G_i$, have some coefficients
$\overline{g}_i$ and the final algebraic
relation for $z$ we call $G_0$ with coefficients
$\overline{g}_0$.

Then the hypothesis implies that, for some shape $\tau$, 
we have the existence in a differential extension of a 
solution to a  suitable differential system $\Sigma_\tau$.
Let's write $DA=0$ as a short-hand for $Da=0$
for each coefficient of $A=\sum a_{ij}X^iY^j$, 
and $A(x, \alpha)=0$ as a shorthand for
$\sum a_{ij}x^i\alpha^j=0$, etc. Then the system
$\Sigma_\tau$ has the form:
$$
DA=0, \quad DB_i=0, i=1,\ldots, k, \quad Dx=1,
\quad A(x, \alpha)=0
$$
$$
Dz=\alpha, \quad 
B_i(x_i, \frac{1}{Dx_i}Dz_i)=0, i=1,\ldots, k,
$$
$$
DG_i=0, i=1,\ldots, k,\quad DG=0,
$$
$$
G_i(x, x_1, \ldots, z_{i-1}, y_i)=0, \quad
G(x, x_1, z_1,\ldots, x_k, z_k, z)=0,
$$
where $y_i=x_i$ or $z_i$ is dictated by the shape.

We can further augment this system with variables 
$c_i, \gamma_i$ for elements as above (here each 
$\gamma_i$ is a tuple of constants giving the
coefficients in a linear combination of elements 
in a basis of the primitives of exact
${\bf m}_i$-differentials) in the linear relation
to get a system $\Sigma^+_\tau$ including in addition
$$
\sum_{i=1}^k c_i(z_i+\gamma_i)=0.
$$

By the Seidenberg Elimination Theorem, the existence 
of solutions (in some differential extension field) 
is characterized by some differential
algebraic constructible systems over $\Z$ on the
non-eliminated variables.

Let us first eliminate all the way down to
$\overline{a}, \overline{b}_i, i=1,\ldots, k$.
Then the existence of a solution $(x, \alpha,\ldots)$
is characterized by some algebraic constructible 
conditions ${\rm Cond}(\overline{a}, \overline{b}_i)$.

Next consider the elimination down to 
$\overline{a}, \overline{b}_i, c_i, \gamma_i$.
As these elements are all constant, this is again a 
constructible algebraic condition over $\Z$, and
the theorem guarantees that suitable $c_i, w_i$ exists
if ${\rm Cond}(\overline{a}, \overline{b}_i)$ holds.
Thus, given $\overline{a}, \overline{b}_i$,
the admissible $c_i$ may be described in constructible
algebraic terms over $\overline{a}, \overline{b}_i$,
and we conclude the rationality statement.

Let us consider the elimination down to
$\overline{a}, \overline{b}_i, c_i, w_i, x$. 
The constructible differential algebraic system 
cannot distinguish the different solutions to 
$Dx=1$, as $x$ can only enter via $Dx$ in any equality
$f(\overline{a}, \ldots, x)=0$, while any
inequality $f(\overline{a}, \ldots, x)\ne 0$
will be automatically satisfied if such $x$ appears. 
So if, for some $\overline{a}, \overline{b}_i$, 
there is a  solution for some $x$ with $Dx=1$ 
then there is a solution for any such $x$. 

Next we include $\alpha$.
The system ${\rm Cond}(\overline{a}, \overline{b_i}, 
c_i, w_i, x, \alpha)$ can't distinguish between roots 
of $A(x,\alpha)$ over $\Q(\overline{a}, \overline{b}_i, x)$,  so if there is a solution for some branch then 
there is one for any branch. 
To see this, we note that if $A(x, \alpha)=0$ then
$D\alpha$ can be expressed as a ratio of polynomials
(depending on $A$) in $x$ and $\alpha$. Thus if
$f(\overline{a}, \ldots, x, \alpha)=0$ holds for 
some root $\alpha$ it holds for all its conjugates
$\alpha'$ over $C(x)$ as well, and likewise for 
$g(\overline{a}, \ldots, x, \alpha)\ne 0$.

Now in the situation of Theorem 3.2 we also conclude 
the existence of an algebraic curve in the base 
$X\times X_1\times\ldots\times X_k$, of suitable 
degrees over $X$ in each $X_i$. 

In general there is no bound on the 
degrees of the $x_i$ over $x$
in terms of the given data, as these
represent the equations of weakly special 
subvarieties.
But, in any case, for some degree there are such 
$x_i$ algebraic over $x$ given by 3.2, and so they 
appear in some elimination system, and can be taken
to be defined by (constructible) algebraic equations 
over the base. This completes the proof.
\end{proof}

There is a more general version of Theorem 5.1
analogous to Theorem 3.2. Here we do not require
a tower of fields with the specific form as in
$B$-strict integrability, but only  a sequence
of tuples $(x_i, z_i)$ with 
$B_i(x_i, \frac{1}{Dx_i}Dz_i)=0$ whose total
transcendence degree is too small.

\medbreak

\begin{thm}[Ax-Schanuel Theorem in a differential field]
Let $K$ be a differential field of characteristic 0 
with constant field $C$ and $x\in K$ with $Dx=1$. 
Let $B_1,\ldots, B_k\in C[X, Y]$ be absolutely irreducible. 

%%Let $\alpha\in K$ with $A(x, \alpha)=0$. 
Suppose that there are elements 
$x_1,\ldots, x_k, z_1,\ldots, z_k$ in $K$ such that
$$
B_i(x_i, \frac{1}{Dx_i}Dz_i)=0
$$
for each $i$ and
$$
{\rm tr. deg.\/}\big(C(\overline{x}, \overline{z})/C\big)\le k.
$$
Then there is a non-empty set $I\subset\{1,\ldots, k\}$,
a differential extension $\tilde{K}$ 
(with constants $\tilde{C}$),
non-zero constants $c_i\in \tilde{C}$,
and algebraic functions $\gamma_i$ 
(primitives of suitable ${\bf m}_i$-exact differentials)
such that
$$
\sum_{i\in I} c_i(z_i+\gamma_i)=0.
$$
Moreover, we can find $x_i^*$ mutually algebraic over each
other and $z_i^*$ with $B_i(x_i^*, \frac{1}{Dx_i^*}Dz_i^*)=0$
for which the same relation holds, and we can take
the $c_i, \gamma_i$ to be defined over an algebraic
extension of the field of definition of the $B_i$. Again,
the relation corresponds to a coset of a suitable algebraic
subgroup of the corresponding product of generalized Jacobians.
\qed
\end{thm}

Then one has a corresponding version of Theorem 1.2, and
a decision procedure for this more general notion of
``weak integrability'' positing that $x_0, z_0$ solves some
``over-determined'' algebraic system on some $x_i, z_i$. 

\medbreak

\noindent{{\bf Proof of Theorem 1.1.\/}
This follows immediately from Theorem 5.1. \qed}{}

\section{Hodge Theoretic perspective on meromorphic 
differentials}

We establish here some well-known and less well-known relations on how to think of meromorphic forms on a curve cohomologically. We let $X$ denote a smooth compact complex curve and $U$ a Zariski-open set of $X$. Recall that $H^1(X,\C)$ has a natural subspace isomorphic to $H^0(X,\Omega_X^1)$ which we will denote by $F^1H^1(X,\C)$. The complex Jacobian $J_X$ can be naturally identified with $H^1(X,\C)/\big(F^1H^1(X,\C)+H^1(X,\Z)\big)$. To relate this to the usual construction of $H^0(X,\Omega_X^1)^\vee$ modulo periods, we simply use the identification $H^1(X,\C)/F^1H^1(X,\C)\cong H^0(X,\Omega_X^1)^\vee$ given by the cup product. 

For this section, $R$ will denote a sub-ring of $\C$, with the key cases being  $\Z,\Q,\C$.

\subsection{Residue-less forms}

For simplicity, and because this will come up as a special case, we first consider the case of meormorphic forms without residues. To understand this space better, let $U=X-D$ be an open set, and let $R\langle D\rangle$ denote the free $R$ module on $D$ and $R^0_D$ the submodule of coefficients that sum to $0$. Then we have a natural exact sequence
\begin{equation}\label{mixedhodge}
0\ra H^1(X,R)\ra H^1(U,R)\ra R^0_D\ra 0.
\end{equation}

Now, we can consider the infinite - dimensional space $H^0(U,\Omega_U)$ which has a natural map to $H^1(U,\C)$. There is a natural subspace $H^0(U,\Omega_U)'$ of forms whose residues at all points of $D$ are 0. Equivalently, this is the kernel of the map to $R^0_D$. By the exact sequence above, there is a natural map $H^0(U,\Omega_U)'\ra H^1(X,R)$. Note that if a form is exact than it is in the kernel of the above map.

\begin{thm}\label{residuelessisom}

If $U\neq X$, then there are isomorphisms
\begin{enumerate}
\item 
$$H^0(U,\Omega_U)/dH^0(U,\Oo_X) \cong H^1(U,R),$$ and 
\item $$H^0(U,\Omega_U)'/dH^0(U,\Oo_X) \cong H^1(X,R),$$
\end{enumerate}
\end{thm}

\begin{proof}

We first prove injectivity. Suppose $\omega$ is in the kernel of the above map. Then The integral of $\omega$ along any closed loop is 0, which means $\omega$ has a well defined integral, or in other words $\omega=dg$ for a meromorphic function $g$ on $U$. Moreover, $g$ is meromorphic since $\omega$ is, and hence is algebraic (by Riemann Existence, for example).

For surjectivity, we simply compute dimensions. We want to work with finite dimensional spaces, so we notice that for each positive integer $m$, 
$H^0(X,\Omega_X(mD))'/dH^0(X,\Oo((m-1)D))\hookrightarrow H^0(U,\Omega_U)'/dH^0(U,\Oo_X)$. Computing dimensions using Riemann-Roch for $m$ large, we see that $$\dim H^0(X,\Omega_X(mD)) = g+m|D|-1,$$ and 
$$\dim H^0(X,\Omega_X(mD))' = g+(m-1)|D|,$$ and also
$$\dim H^0(X,\Oo((m-1)D)) = (m-1)|D|+(1-g);$$  and since the kernel of $d$ is 1-dimensional, 
$$\dim dH^0(X,\Oo((m-1)D)) = (m-1)|D|-g.$$ Hence, finally,
$$\dim \big(H^0(X,\Omega_X(mD))/dH^0(X,\Oo((m-1)D))\big) = |D|+2g-1,$$ which gives surjectivity of the first map, and 
$$\dim \big(H^0(X,\Omega_X(mD))'/dH^0(X,\Oo((m-1)D))\big) = 2g,$$ which gives surjectivity of the second map. 
\end{proof}

\subsection{Forms with Residues}

We let $U=X-D$ as in the previous subsection. Now we set $F^1H^1(U,\C)$ to be the image of $H^0(X,\Omega_X(D))$ - in other words, differential forms which have simple residues along $D$. We

\begin{lemma}\label{quotmap}

$H^1(U,\C)/F^1H^1(U,\C)\cong H^1(X,\C)/F^1H^1(X,\C)$

\end{lemma}

\begin{proof}

Riemann-Roch shows that $F^1H^1(U,\C)$ surjects onto $\C^0_D$. The claim thus follows by computing dimensions.
\end{proof}

We thus obtain a canonical map $\phi_U:H^1(U,\Z)\ra H^1(X,\C)/F^1H^1(X,\C)$ by composing the map in lemma \ref{quotmap} with the natural map
$H^1(U,\Z)\rightarrow H^1(U,\C)/F^1H^1(U,\C)$. Tensoring up with $\C$ we obtain a map
$$\phi_{U,\C}:H^1(U,\C)\ra H^1(X,\C)/F^1H^1(X,\C) \otimes_\Z \C.$$ Finally, taking a direct limit gives a map
$$\phi_\C: \varinjlim_U H^1(U,\C)\ra H^1(X,\C)/F^1H^1(X,\C) \otimes_\Z \C.$$

Note that $\frac{df}{2\pi i f}$ is in the kernel of $\phi_U$  and hence of $\phi_\C$. Let $\Omega_{X,d\log}$ denote the complex vector space spanned by all such differentials, which naturally injects into $\varinjlim_U H^1(U,\C)$.

\begin{thm}\label{residueisom}

The map $\phi_\C$ induces an isomorphism between $\frac{\varinjlim_U H^1(U,\C)}{\Omega_{X,d\log}}$ and $H^1(X,\C)/F^1H^1(X,\C) \otimes_\Z \C.$

\end{thm}

\textbf{Notation: In the following, we write
$R^{X(\C)}$ to mean the $R$-module of functions $X(C)\ra R$ with finite support, and  $R_0^{X(\C)}$
to mean the submodule of functions whose sum over all values is 0.}

\begin{proof}

We first establish surjectivity. We first observe the exact sequence 
$$0\ra H^1(X,\Z)\ra H^1(X,\C)/F^1H^1(X,\C) \ra J_X\ra 0$$ where $J_X(\C)$ is the complex points of the Jacobian of $X$. We thus get induced maps 
$\widetilde{\Phi_U}:\Z_0^D\ra J_X(\C)$ and $\widetilde{\Phi}:\Z_0^{X(\C)}\ra J_X(\C)$ which is the usual map sending degree 0 divisors to the Jacobian \cite[Prop 12.7]{Voisin}. This is well known to be surjective, and therefore $\C_0^{X(\C)}$ surjects onto $J_X(\C)\otimes_\Z\C$. This is equivalent to $\phi_\C$ being surjective after quotienting out by $H^1(X,\C)$. However, $\phi_\C$ induces an isomorphism $H^1(X,\C)\ra H^1(X,\Z)\otimes_\Z\C=H^1(X,\C)$. Thus we have established surjectivity.

Next we establish injectivity. Suppose that $\gamma\in \varinjlim_U H^1(U,\C)$ is in the kernel of $\phi_\C$. Then as we established above, the image $\tilde{\gamma}\in\C_0^{X(\C)}$ in $J_X(\C)\otimes_\Z\C$ is $0$. Since tensoring with $\C$ over $\Z$ is exact, this means that $\tilde{\gamma}$ is a $\C$-linear combination of principal divisors $(f)$. Since $\tilde{\frac{df}{2\pi if}}=(f)$, we may subtract an element $\tau$ of $\Omega_{X,d\log}$ from $\gamma$ to obtain an element $\gamma'$ in the kernel of $\phi_U$ whose image in $\C_0^D$ is $0$. However, that means that $\gamma'\in H^1(X,\C)$, and since ${\phi_U}\mid_{H^1(X,\C)}$ is injective, we must have that $\gamma'=0$. Thus $\gamma=\tau\in\Omega_{X,d\log}$ as desired.
\end{proof}

Finally, we obtain the following:

\begin{thm}\label{formsisom}
Let $M^1_X$ denote the group of all meromorphic 1-forms on $X$, and $M_X$ denote the field of  meromorphic functions on $X$. Then
$$\frac{M^1_X}{dM_X+d\log M^*_X}\cong H^1(X,\C)/F^1H^1(X,\C) \otimes_\Z \C.$$

\end{thm}

\begin{proof}
Taking the direct limit over $U$ of the first isomorphism
in Theorem \ref{residuelessisom}, we see that
$\frac{M^1_X}{dM_X}\cong \varinjlim_U H^1(U,\C)$. 
The result now follows from Theorem \ref{residueisom}.

\end{proof}

\begin{remark}
From a Hodge theory perspective, the $\Z$ extension
\eqref{mixedhodge} along with the $\C$-subspace
$F^1H^1(U,\C)$ constitutes the mixed Hodge structure 
on $H^1(U,\C)$. The forms $\frac1{2\pi i}d\log(f)$ are
precisely the Hodge vectors. The extension class is
determined by the subgroup generated by $\Z^0_D$ in 
the Jacobian. The universal extension is then naturally
constructed the direct limit $\lim H^1(U,\C)$ 
quotiented out by all trivial extensions, which are 
just the $d\log(f)$. Thus, the above result is 
essentially saying that 
$$H^1(X,\Z)\ra H^1(X,\C)/F^1H^1(X,\C)\ra J_X
$$ 
is the universal extension by powers of $\Z(-1)$ 
of the pure Hodge structure on $H^1(X,\C)$.
\end{remark}

\section{Decision Procedures}

The purpose of this section is to provide an algorithm to fully solve the decision problem of whether one period-function can be expressed by finitely many others others.

\subsection{Algorithms from algebraic-geometry}

Following \cite{PTV} we use finitely generated fields (fields of finite type) but we restrict to characteristic 0. We work over the algebraic closure of a finite type field $K$, presented as he fraction field of a finitely-generated $\Z$-algebra. 

The following result is well-known to experts but we couldn't find a clean reference, so we record it in the form we need:

\begin{lemma}\label{lem: Computing Mordell-Weil Groups}

Let $A/\ol{K}$ be an abelian variety, and let $\phi:\Z^n\ra A(\ol K)$ be a homomorphism, presented by specifying the image of a basis. Then there is a decision procedure which returns the image $\phi(\Z^k)$ as $F+T$ where $T$ is a torsion subgroup and $F$ is a free abelian group. Equivalently, the decision procedure returns generators for the kernel of $\phi.$

\end{lemma}

\begin{proof}

First, by increasing $K$ we may assume that $A$ and $\phi$ are defined over $K$ itself.

Next, assume $K$ is a number field. Then using the theory of the N\'eron-Tate height we can compute the torsion subgroup $A(K)_{tor}$ of $A(K)$. Set $A(K)':=A(K)/A(K)_{tor}$. if we compute the kernel of $\psi:\Z^n\ra A(K)'$ then finding the kernel of $\phi$ is simply a matter of checking the finitely many sublattices of $\ker\psi$ finite index at most 
$\#A(K)_{tor}$. We thus focus on finding $\ker\psi$.

Moreover, again using the theory of the N\'eron-Tate height we may compute the image of $\psi$ in $A(K)'/mA(K)'$ for every positive integer $m$, as this just amounts to checking which of finitely many elements are $m$'th powers in $A(K)$. Note that if $P_1,\dots,P_d$ are independent in $A(K)'$ then they are independent in $A(K)/pA(K)$ for a large enough prime $p$.  Thus, by day we can look for elements in the kernel of $\phi$, and by night we can try to prove it by showing independence modulo large primes $p$.

Finally, assume $K$ is a finite type field. Recall we may write $K$ as the fraction field of a finitely-generated $\Q$-algebra $R$. We may assume $A$ spreads out over $R$ such that $A(K)=A(R)$. Then by the main theorem of \cite{MASSER} there are specializations $R\ra L$ where $L$ is a number field which induce isomorphisms $A(R)\ra A(L)$. Thus, we may by day look for elements in $\ker\phi$ and by night try to prove we've found them all by computing the same for all specializations.
\end{proof}

\subsection{Rephrasing through traces}

\begin{definition*}

Let $(X_i,\omega_i), i=1,2$ be curves with
differentials. Let $D\subset X_1\times X_2$ 
be a curve with no fibral-components. Then if
$\omega_2=\pi_{2*}\pi_1^* \omega_1$ we say that
$\omega_2$ is a \emph{trace-image} of 
$\omega_1$. 

\end{definition*}

\begin{lemma}\label{lem:trace-version}
Let $(X_i,\omega_i), 0\leq i\leq m$ be curves with meromorphic differentials over $\ol K$. TFAE:

\begin{itemize}
    \item $(X_0,\omega_0)$ is integrable in terms of $(X_1,\omega_1)$
    \item $\omega_0$ is in the linear span of differentials on $X_0$ which are trace-images of the $\omega_i$, and exact differentials.
\end{itemize}

\end{lemma}
\begin{proof}

By Theorem \ref{thm:ei} condition (i) is equivalent to the existence of a curve $C$ with maps $\phi_j$ to the $C_{f(j)}, j>0$   and a single map $f$ to $C_0$ such that $f^*\omega_0$ is in the linear span of the $\phi_j^*\omega_{f(j)}$, for some function $f$. Assuming this is the case, taking the trace from $C$ to $X_0$ yields (ii).

Now suppose that (ii) holds, so that there are correspondences $D_i\subset X_0\times X_{f(i)}$ such that the push-pull of $\omega_{f(i_)}$ contain $\omega_0$ in their linear span. 

Let $d_i$ be the degree of $d_i$ over $X_0$. We first let $E_i$ be the maximal closed reduced subscheme of the $d_i$ fiber product of $D_i$ over $X_0$ whose generic points admit  $d_i$ distinct maps to $D_i$. We then let $C$ be the fiber product of all of the $E_i$ over $X_0$. Note that for each $i$ , $C$ has $d_i$ distinct maps to $C_{f(i)}$. Moreover, there is an action of $G:=\prod_i S_{d_i}$ action on $C$ making it a torsor over $X_0$. We may thus identify forms on $X_0$ with $G$-invariant forms on $C$.

Now consider a fixed $i$, and the form $\pi_{2*}\pi_1^*\omega_{f(i)}$. If we pull this form back to $E_i$, it becomes $S_{d_i}$-invariant and may be expressed as a sum over $\pi^*\omega_{f(i)}$ where $\pi$ ranges over all $d_i$ maps to $X_0$. Thus the same is true once this form is pulled back all the way to $E$. 

Finally, we simply write the linear relation guaranteed by (ii) and pull it back to $E$, obtaining (i). Note that our curve $E$ is not irreducible, but we may restrict to any irreducible component.
\end{proof}

\subsection{Decision procedures}

We are now ready to give an algorithmic procedure for deciding whether $(X_0,\omega_0)$ is integrable in terms of the $(X_i,\omega_i)$. The procedure deals with residue-less forms and forms-with-residues separately, though the spirit is similar.

\subsubsection{\bf None of the $\omega_i, i>0$ have any residues}\!

\medskip
{\bf Step 1: Reduction to computing a basis for trace images}

%In this case, if $\omega_0$ has a residue then it is clearly not integrable in terms of the $\omega_i$, as pull-backs preserve the property of being residue-less. Thus we may assume $\omega_0$ is also residue-less.
We claim that it is sufficient to obtain a basis for the linear span of trace-images of each $\omega_i$ for each $i$, modulo exact differential forms. Indeed, suppose we have obtained such a basis $L$. Then by lemma \ref{lem:trace-version} it is sufficient to decide whether $\omega_0$ is in the linear span of $L$ and exact-forms. Now, suppose ${\bf m}$ is a large enough modulus such that $\omega_0$ and $L$ have their poles in ${\bf m}$. Then it is enough to restrict to exact-differential forms whose poles are in ${\bf m}$. These are finite dimensional vector spaces that can be easily computed, and the decidability question can be answered from here.

%\subsubsubsection
{\bf Step 2: Computing correspondences between curves}

By Theorem \ref{residueisom} there is a natural isomorphism between reside-less forms modulo exact ones and $H^1$ of the curve. These isomorphisms are just given by integration alon cycles and so are compatible with push-forwrads and pull-backs.
Thus, given a curve $D\subset X_0\times X_1$ we obtain a natural map $\phi_D:H^1(X_1,\C)\ra H^1(X_0,\C)$, and the image of $[\omega_1]$ under this map will be the class of the trace-image of $\omega_1$ under $C$. Hence, it is enough to obtain a basis for all maps $\phi_D$ corresponding to all correspondences $D$.

This is done by Theorem 8.15 of
\cite{PTV}
%https://math.mit.edu/~poonen/papers/compute_ns.pdf.

The above returns a set of cycles $D_i$, which 
finishes the algorithm. 

%\subsubsection

\subsubsection
{\bf At least one of the $\omega_i, i>0$ 
has a non-zero residue}\!\medskip

%\subsubsubsection
{\bf Step 1: Getting $\log(x)$}

We first show that we have access to the logarithm function. Suppose wlog that $\omega_1$ have at least one residue. We construct a map $f:C\ra \PP^1$ which sends all the points at which $\omega_1$ has a residue to $0$ except one which gets sent to $\infty$. Then the trace of $\omega_1$ will be some multiple of $\frac{dx}{x}$ and so by summing up the primitives we obtain the logarithm function. Crucially, this gives us access to all differentials of the form $d\log f$ where $f$ is a meromorphic function. 

%\subsubsubsection
{\bf Step 2: Reduction to Computing a basis for trace images}

We now proceed very much like before. We claim that it is sufficient to obtain a basis for the linear span of trace-images of each $\omega_i$, modulo exact and log-exact differential forms. Indeed, suppose we have obtained such a basis $L$. Then Theorem \ref{lem:trace-version} says that it is sufficient to check whether $\omega_0$ is in the span of $L$, and exact and log-exact differential forms. 

Now, by Theorem \ref{residueisom}, we may work in $H^1(X_0,\C)/F^1H^1(X_0,\C)\otimes_\Q \C$. So we may first take the image in $J_0\otimes_\Q\C$. Thus $L$ and $\omega_0$ gives us a morphism $G:\Z^M\ra J_0\otimes_\Q\C $. Moreover, this morphism is very explicit, given by just taking the residues of our forms. By taking bases for the residues over $\Q$, this becomes a question of identifying the subgroup in $J_0$, which is Lemma \ref{lem: Computing Mordell-Weil Groups}. Thus we may obtain the kernel of $G$. 

In that case that the image of $\omega_0$ in  $J_0\otimes_\Q\C$ is already not expressible in terms of $L$, we can stop. Otherwise, by subtracting off the appropriate element in $L$, we may assume the image of $\omega_0$ is $0$, and we may restrict to the subgroup $L'$ of $L$ in the kernel of $G$. Now adding log-exact forms will introduce a non-zero residue, and so we are reduced to checking whether $\omega_0$ is in the span of $L'$ and exact forms, which we do as in the previous subsubsection.

\subsection{Elliptic integrals}

In this section we explain how to generalize the above algorithm the following question: \emph{When is a period-function expressible in terms of elliptic integrals?}

Recall that an elliptic integral is defined as $$f(x)=\int_c^x R(t,\sqrt{P(t)})dt$$ where $P(t)$ is a polynomial of degree 3 or 4, and $R$ is a rational function. If we consider the elliptic curve $E_P:={y^2=P(x)}$ then $\omega=R(x,y)dx$ is a rational differential form on $E_P$, and its integral is precisely the pullback of $f(x)$. Thus, in answering the above question, we have the following:

\begin{thm}\label{ellipticdecision}
Let $(X,\omega)$ be a curve with a rational differential form. Then there is a decision procedure for determining where $(X,\omega)$ is integrable in terms of the set of all elliptic curves with all rational differentials on them.
\end{thm}

\begin{proof}

We proceed in as in the previous two sections. Note that $\omega$ gives us a class $[\omega]$ in $H^1(X,\C)/F^1H^1(X,\C)\otimes_\Q\C$ by Theorem \ref{residueisom}. The key to our proof is first decomposing the Jacobian $J_X$ into its elliptic and non-elliptic part. Note that if $J_X$ contains an elliptic curve then there is a map from $X$ to that elliptic curve. Thus we must find all elliptic curves with a map from $X$. These all show up inside $X\times X$ via their induced correspondences, and thus we may find them all as before by computing the Neron-Severi group of $X\times X$. Having found all maps from $X$ to elliptic curves, we may write $J_X\sim A\times B$ via an explicit isogeny such that $A$ is a power of elliptic curves and $B$ does not have elliptic factors. Note that this induces natural isomorphisms
$$H^1(X,\C)/F^1H^1(X,\C)\cong \tilde A\times \tilde B$$ where $\tilde A$ denotes the universal cover, and also $$H^1(X,\C)/F^1H^1(X,\C)\otimes_\Q\C \cong \tilde A\otimes_\Q\C\times \tilde B\otimes_\Q\C.$$

Finally, the question is whether $[\omega]$ has trivial image in $\tilde B\otimes_\Q\C$. As before, we first check whether the image of $[\omega]$ is $0$ in $B\otimes_\Q\C$. This is merely a question about the Mordell-Weil group of $B$ and can therefore be answered as before.

Supposing this is the case, we then adjust $\omega$ by the image of an elliptic form so that $[\omega]$ has trivial image also in $\tilde A\otimes_\Q\C$ and thus in $J_X\otimes_\Q\C$ (and is therefroe residueless). Thus, $[\omega]$ is now merely a class in $H^1(X,\C)$ and we must check if its contained in $H^1(A,\C)$. But this is easy, as we may generate a basis for $H^1(A,\C)$ using correspondences of $X$ with elliptic curves. The proof is therefore complete.

\end{proof}

\begin{remark}
\begin{enumerate}
    \item Note that in the above proof we handled the case of all elliptic integrals, but it is easy to adjust the proof so as to allow only consider $(E,\omega)$ where $\omega$ is regular, or only having simple poles, or being residue-less. This would correspond to periods of differentials of the first, second, or third kind. 
    \item Since (almost) every Abelian variety of dimension 2 or 3 is the Jacobian of a curve, one may similarly make an argument for the family of all differentials on all genus $g$ curves, where $g=2,3$.
\end{enumerate}
\end{remark}

% Now, by the Lefschetz (1,1) Theorem, the integral span of the $\phi_D$ correspond precisely to the Hodge morphisms between $H^1(C_1)$ and $H^1(C_0)$, which in turn correspond to all morphisms between the Jacobians of $C_1$ and $C_0$ by looking at the map induced on the tangent space. Hence, we are reduced to computing a basis for $\Hom(J_1,J_0)\otimes_\Q$.

% \subsubsubsection{step 3:Computing a basis for morphiss}

% Should be easy, but I'm currently feeling lazy...

\section{A Connection with unlikely intersections}

For a fully general pencil of curves and
differentials, Masser and Zannier \cite{MASSERZANNIER} 
are able to precisely (and effectively) describe 
when the generic fibre in the pencil is not 
elementary integrable, but that  elementary 
integrabililty of fibres is not ``unlikely'' in 
the  Zilber-Pink sense (and indeed occurs for 
infinitely many values of the parameter). Further, 
in the ``unlikely'' case, they are able to prove 
the finiteness statement via their results on 
Relative Manin-Mumford.

Here we restrict to regular differentials and observe
that the question of exceptional integrability
again leads to questions of Zilber-Pink type, 
and that results in the literature answer them 
in some cases.

\begin{thm}
Fix some smooth projective curve $Y$ of genus 
$g\ge 1$ and a non-zero regular differential $\eta$, and 
suppose the Jacobian ${\rm Jac}(Y)$ is simple. 

Suppose that $(X_t, \omega_t)$ is a pencil of curves 
of genus $g$ and regular differentials over some
quasi-projective base curve $B$ (so we have removed
finitely many points where the fibre $X_t$ is not 
smooth or the differential not regular).
We can assume that $B\subset \mathcal{A}_g$. 

Then there are only finitely many $t$ such 
that $z_t=\int^{x_t}\omega_t$ is
$\{w=\int^y\eta\}$-integrable if $B$ is not  a 
weakly special subvariety.
\end{thm}

The same holds (with same proof) if we assume 
several differentials on $Y$ are available, on 
several different (simple Jacobian) curves $Y$.

\begin{proof}
The integrability condition entails that there  
is a non-product weakly special subvariety of 
$${\rm Jac}(X_t)\times {\rm Jac}(Y)^n$$
which is
dominant to both factors. But $Y$ is simple, so this can only happen if
$X_t$ is isogenous to $Y$. If there are infinitely many such 
specialisations we find that $B$ has infinitely many points in the 
isogeny class of the moduli point $[Y]\in\mathcal{A}_g$ of $Y$. 
This is a problem of Zilber-Pink type, more specifically
of Andr\'e-Pink-Zannier type: one expects
this can only happen if $B$ is a proper weakly special subvariety,
and for curves this is a theorem of Orr 
\cite[Theorem 1.2]{ORR}. 
\end{proof}

Note that if $X_t$ is isogenous to $Y$ this does 
not in general lead to integrability (the 
differentials might not satisfy a suitable 
linear relation), so our theorem is not sharp. 
But isogeny does lead to integrability in the 
case $g=1$, since elliptic curves have only 1 regular differential (up to scale). However Theorem 8.1 is uninteresting 
in that case as the moduli space is one-dimensional, 
so $B$ is always weakly special.

We can consider however the question of when two 
given elliptic logarithms $z_1=\int^{x_1}\omega_1$
on an elliptic curve $X_1$ and $z_2=\int^{x_2}\omega_2$
on an elliptic curve $X_2$ are simultaneously 
integrable by means of a third elliptic logarithm
$z_3=\int^{x_3}\omega_3$ on an elliptic curve $X_3$. 
If now $X_1, X_2, X_3$ vary in a pencil then
we may assume that the pencil is parameterised 
by the points of a curve $V\subset Y(1)^3$. 
Then the Zilber-Pink conjecture predicts that 
the set of $t\in V$ for which two functions are 
integrable in terms of the third (i.e. the three 
elliptic curves are pairwise isogenous) is finite 
unless $V$ is contained in a proper special 
subvariety. Some partial results on this problem 
are in \cite{HABEGGERPILA}.

\bigbreak
\noindent
{\bf Acknowledgements.\/} The authors thank
Daniel Bertrand and David Masser for enlightening correspondence.

\bigbreak

\bigskip

\noindent
\leftline{JP: Mathematical Institute, 
University of Oxford, Oxford, UK.}
\rightline{pila@maths.ox.ac.uk}

\bigskip

\noindent
\leftline{JT: Department of Mathematics, 
University of Toronto, Toronto, Canada.}
\rightline{jacobt@math.toronto.edu}

\vfil
\eject

\end{document}